\documentclass[11pt,a4paper]{amsart}

\usepackage[a4paper,textwidth=6.25in,textheight=8.7in,centering]{geometry}
\usepackage{amsmath,amssymb,amsthm,mathtools,mathrsfs}
\usepackage{bbm}
\usepackage{enumitem}
\usepackage[section]{placeins}
\usepackage{tikz}
\usetikzlibrary{arrows.meta,positioning}
\usepackage{xcolor}
\colorlet{RED}{red}
\usepackage[normalem]{ulem}
\usepackage[colorlinks=true,linkcolor=black,citecolor=blue,urlcolor=blue]{hyperref}

\usepackage{marginnote}
\newtheorem{theorem}{Theorem}[section]
\newtheorem{proposition}[theorem]{Proposition}
\newtheorem{lemma}[theorem]{Lemma}
\newtheorem{corollary}[theorem]{Corollary}
\theoremstyle{definition}

\newtheorem{problem}[theorem]{Problem}
\newtheorem{example}[theorem]{Example}
\theoremstyle{remark}

\newcommand{\R}{\mathbb R}

\newcommand{\norm}[1]{\left\lVert #1\right\rVert}

\newcommand{\conv}{\operatorname{conv}}

\newcommand{\dens}{\operatorname{dens}}
\newcommand{\htop}{\operatorname{ht}}

\newcommand{\sstar}{\subsetneq^{*}}
\newcommand{\ssubstar}{\subseteq^{*}}

\newcommand{\cl}[1]{\overline{#1}}
\newcommand{\map}[3]{#1:#2\to #3}
\newcommand{\set}[2]{\left\{#1\;:\;\,#2\right\}}
\newcommand{\st}{($*$)}
\newcommand{\tn}[1]{\left|\kern-.9pt\left|\kern-.9pt\left|#1\right|\kern-.9pt\right|\kern-.9pt\right|}
\newcommand{\tndot}{\left|\kern-.9pt\left|\kern-.9pt\left|\cdot\right|\kern-.9pt\right|\kern-.9pt\right|}

\title[The renorming-stable ball covering property and $C(K)$-spaces]{The renorming-stable ball covering property and $C(K)$-spaces}
\author{Rui Liu}
\address{School of Mathematical Sciences and LPMC,
Nankai University,
Tianjin 300071,
P.R. China}
\email{ruiliu@nankai.edu.cn}

\author{Jie Shen}
\address{School of Mathematical Sciences and LPMC,
Nankai University,
Tianjin 300071,
P.R. China}
\email{1710064@mail.nankai.edu.cn}

\author{Richard J. Smith}
\address{School of Mathematics and Statistics,
University College Dublin,
Belfield,
Dublin 4,
Ireland}
\email{richard.smith@maths.ucd.ie}

\thanks{Rui Liu and Jie Shen were partially supported by the National Natural Science Foundation of China (Grant No. 12471131).}
\date{}

\begin{document}

\begin{abstract}
A Banach space has the ball-covering property (BCP) if its unit sphere can be covered by countably many balls that miss the origin. A Banach space has the renorming-stable ball-covering property (RBCP) if it has the BCP with respect to every equivalent norm. We investigate the RBCP from the perspective of $C(K)$-spaces. In particular, we show that there are many non-separable $C(K)$-spaces having the RBCP. This yields an emphatic negative answer to a question of Cheng, Kato and Zhang from 2020.
\end{abstract}

\maketitle

\section{Introduction}

The ball-covering property (BCP) concerns countable coverings of the unit sphere by balls which avoid the origin. Let \((X,\norm{\cdot})\) be a normed space, write \(S_X=\{x\in X:\norm{x}=1\}\), and put \(B(x,r)=\{y\in X:\norm{y-x}<r\}\). The space has the BCP if there are \(x_n\in X\) and \(r_n>0\), \(n\in\omega\), such that
\[
S_X\subseteq\bigcup_{n\in\omega}B(x_n,r_n)\qquad\text{and}\qquad 0\notin B(x_n,r_n),\quad n\in\omega.
\]
It has the uniform ball-covering property (UBCP) if such a covering can be chosen so that
\[
\sup_{n\in\omega}r_n<\infty\qquad\text{and}\qquad\inf_{n\in\omega}\bigl(\norm{x_n}-r_n\bigr)>0.
\]
Thus the radii are uniformly bounded and the covering balls stay uniformly away from the origin. This terminology can be found in existing ball-covering literature \cite{Cheng2006,Cheng2011,LuoZheng2021}. 

Ball-covering properties are an active topic of research. Recent work has linked ball-covering conditions to renorming characterizations of Banach spaces containing \(\ell_1(\kappa)\) \cite{CLL2023,AMRZ2023}. In spaces of operators, approximation properties have yielded quantitative estimates for the UBCP under a family of equivalent norms \cite{BLS2027}, while failure of the BCP has been established for Calkin algebras under unconditional basis assumptions \cite{SZ2026}. Further developments include geometric tests for the BCP of operator spaces \cite{MPSS2026} and sufficient conditions for the UBCP of dual spaces \cite{LMS2026}. Ball-covering properties have also been studied in Lipschitz function spaces \cite{LMS2026,MPSS2026} and Lipschitz-free spaces \cite{APPP2026}.

It is an easy consequence of the Hahn-Banach separation theorem that if $X$ has the BCP then $(X^*,w^*)$ is separable \cite[Proposition 4.1]{Cheng2006}. Conversely, if \((X^*,w^*)\) is separable, then for every \(\varepsilon>0\), \(X\) admits a \((1+\varepsilon)\)-equivalent norm having the (uniform) BCP \cite{FonfZanco2009, CSZ2009}. Thus \(w^*\)-separability of the dual characterizes the existence of an equivalent BCP norm.
However the BCP is an isometric rather than an isomorphic property: $\ell_\infty$ has the BCP with respect to the canonical supremum norm, but not with respect to every equivalent norm \cite{ChengChengLiu2008}. 
It also follows from \cite{BLS2027,SZ2026} that the space \(\mathcal{B}(X)\) of bounded linear operators on $X$ admits an equivalent norm without the BCP whenever \(X\) has a shrinking 1-unconditional basis. This motivates us to ask when the BCP persists for every equivalent norm. We say that a Banach space \(X\) has the renorming-stable ball-covering property (RBCP) if every equivalent norm on \(X\) has the BCP. It is obvious that every separable Banach space has the RBCP, so it is only meaningful to ask whether $X$ has the RBCP if $X$ is non-separable.

It is natural to investigate this question relative to the class of Banach spaces of the form $C(K)$, where $K$ is a compact Hausdorff space. This is what we do in this paper. The question is non-trivial because $C(K)$ has the RBCP whenever $K$ is metrizable, and $C(\beta\omega)$ (where $\beta\omega$ is the Stone-\v{C}ech compactification of the set $\omega$ of natural numbers) fails the RBCP because, of course, it is a classical fact that this space is linearly isometric to $\ell_\infty$. Moreover, we should point out that the BCP and UBCP in $C(K)$-spaces (with respect to the canonical supremum norm) turn out to be equivalent, and have been characterized in terms of the topology on $K$:~$C(K)$ has these properties if and only if $K$ admits a countable \(\pi\)-basis of \(K\), that is, there exists a countable family of non-empty open sets such that every non-empty open set contains one of them \cite[Theorem~2.1]{LiuLiuLuZheng2022}. Thus, our paper can be seen as a continuation of this work. 

The paper is organised as follows. Section \ref{sec:general_results} consists of some general positive results.  Our first main result, Theorem \ref{thm:main-thm-A}, shows that if $X$ is a G\^ateaux differentiability space (GDS) and has the countable separation property (CSP) introduced by Talponen in \cite{Talponen2009}, then $X$ has the RBCP. As is shown in Section \ref{sec:examples}, $X$ can have these properties and be very far from separable. Thus we can provide an emphatic negative answer to \cite[Question~5.49]{ChengKatoZhang2020}, which asks whether every space having the RBCP must be separable. While this question first appeared explicitly in print in 2020, it had circulated informally since the foundational works of Cheng \cite{Cheng2006} and of Fonf and Zanco \cite{FonfZanco2009}. Theorem \ref{thm:main-thm-A} leads to Corollary \ref{cor:standard-criteria} and Theorem \ref{thm:rn-characterization}. Corollary \ref{cor:standard-criteria} shows that, for a GDS $X$ having Corson's property (C), having the RBCP and the BCP are both equivalent to the separability of $(X^*,w^*)$. Theorem \ref{thm:rn-characterization} shows that, for totally disconnected Radon-Nikod\'ym compact spaces $K$ having countable tightness, if $K$ is separable then $C(K)$ has the RBCP. Having in mind \cite[Theorem~2.1]{LiuLiuLuZheng2022}, it follows that, for such $K$, the space $C(K)$ has the RBCP, the BCP or the UBCP if and only if $K$ has a countable $\pi$-basis or, equally, if $K$ is separable.

In Section \ref{sec:examples}, we provide a host of examples of non-separable $C(K)$ spaces having the RBCP. In Example \ref{ex:almost_disjoint}, we use Theorem \ref{thm:rn-characterization} to show that if $K$ belongs to the class of Isbell--Mr\'owka spaces constructed in the standard way using an uncountable almost disjoint family of subsets of $\omega$ (see e.g.~\cite{HernandezHernandezHrusak2018} for a survey of such spaces), then $C(K)$ is a non-separable space having the RBCP (see Example \ref{ex:almost_disjoint}). We go on to explore the extent to which $C(K)$ can fail to be separable while still retaining the RBCP, with an emphasis on Asplund spaces that behave badly with respect to renormings. According to Corollary \ref{pr:star}, again by Theorem \ref{thm:rn-characterization}, if $K$ is any separable totally disconnected Radon-Nikod\'ym compact space having property {\st} (see \cite{OrihuelaSmithTroyanski2012}), then $C(K)$ has the RBCP. This implies that if $K$ is separable and scattered, then $C(K)$ has the RBCP if $K$ has countable Cantor-Bendixson height or is (more generally) a $\sigma$-discrete or Gruenhage space. We extend the standard method for embedding the compact uncountable ordinal $\omega_1+1$ in a separable space (see e.g. \cite{Nyikos1997}) to show that the same can be done to the 1-point compactification of any (Hausdorff) tree of height at most $\omega_1$ and cardinality at most $\mathfrak{c}$. We use this to show, in Example \ref{ex:strictly_convex}, that there exists an Asplund space $C(K)$ that has the RBCP but fails to admit any equivalent norm whose dual norm is strictly convex. We note further, in Example \ref{ex:Brech_Koszmider}, that (under a forcing model) the compactum $K$ of Brech and Koszmider gives an Asplund space \(C(K)\) of density \(\aleph_2\) with the RBCP, no equivalent Fr\'echet smooth norm, and no uncountable biorthogonal system \cite{BrechKoszmider2011}. Finally, Corollary \ref{co:Moors-Somasundaram} demonstrates that $C(L_A)$ has the RBCP, where $L_A$ is the compact line of Moors and Somasundaram used to solve the long-standing problem of whether every GDS is weak Asplund \cite{moors:somasundaram:06}. Therefore, it is possible for a Banach space to have the RBCP and not even be weak Asplund, let alone separable.

Finally, in Section \ref{sec:remarks_problems}, we present two open problems prompted by this investigation, along with some accompanying remarks. Throughout the paper, we shall introduce notation and terminology as and when it is needed.
\section{General results}\label{sec:general_results}

Before the main results we recall some definitions and general theory. A Banach space is called an Asplund space, a weak Asplund space or a G\^ateaux differentiability space (GDS) if every continuous convex function defined on it is Fr\'echet differentiable at the points of a dense $G_\delta$ subset, G\^ateaux differentiable at the points of a dense $G_\delta$ subset, or G\^ateaux differentiable at the points of a dense subset, respectively. See e.g.~\cite{fabian:97} for a comprehensive introduction to these properties. Every topological space under consideration in this note will be assumed to be Hausdorff; we make explicit mention of the Hausdorff property only when necessary. We recall that, given a compact space $K$, the space $C(K)$ is Asplund if and only if $K$ is scattered, which means that every non-empty subset of $K$ contains a relatively isolated point (see e.g.~\cite[Theorem 1.1.4]{fabian:97}). Given a Banach space $X$, a subset $E \subseteq X^*$ is called total if, given $x \neq 0$, there exists $f \in E$ such that $f(x)\neq 0$. The existence of a countable total subset of $X^*$ is equivalent to the separability of $(X^*,w^*)$ \cite[Exercise 3.85]{FabianEtAl2011}. We denote by $\omega$ and $\omega_1$ the set of natural numbers and countable ordinals, respectively.

Let $N$ be an equivalent norm on a Banach space $X$. Recall that $j \in B_{(X,N)^*}$ is called a $w^*$-exposed point of $B_{(X,N)^*}$ if there exists $u \in B_{(X,N)}$ such that $j(u)=1$ and $j'(u)<1$ whenever $j' \in B_{(X,N)^*}\setminus\{j\}$. Any such point $u$ is said to expose $j$. We denote by $J_N$ the set of all $w^*$-exposed points of $B_{(X,N)^*}$. By the classic lemma of \v{S}mulyan (\cite[Corollary 7.20]{FabianEtAl2011}), if $u$ exposes $j \in J_N$ then $u$ is a smooth point of $N$, that is, a point of G\^ateaux differentiability of $N$; conversely, if $u \in B_{(X,N)}$ is a smooth point of $N$ then $u$ exposes some (unique) $j_u \in J_N$. We denote by $G_N$ the set of smooth points of $N$ that belong to $B_{(X,N)}$.

Our first result is essentially \cite[Theorem 0.2]{Cheng2011} relative to a given equivalent norm $N$ on $X$. We provide a short proof only to take the opportunity to correct a small gap in the proof of \cite[Theorem 2.6]{ChengChengShi2009}, on which \cite[Theorem 0.2]{Cheng2011} depends.

\begin{proposition}[{cf.~\cite[Theorem 0.2]{Cheng2011}}]\label{lem:fixed-norm-characterization}
Let \(N\) be an equivalent norm on the Banach space \(X\). If \(J_N\) contains a countable total subset, then \(N\) has the BCP. If, in addition, $G_N$ dense in \(S_N\), then the converse also holds. Consequently, a GDS \(X\) has the RBCP if and only if, for every equivalent norm \(N\), the corresponding set \(J_N\) contains a countable total subset.
\end{proposition}

\begin{proof}
 That $N$ has the BCP whenever $J_N$ contains a countable total subset is the sufficiency part of the proof of \cite[Theorem 0.2]{Cheng2011} (in the statement of this theorem $X$ is a GDS, but this property is not needed for this implication). Now assume $G_N$ is dense in $S_N$ and let $\set{B_N(x_i,r_i)}{i \in \omega}$ be a BCP covering. If necessary, by replacing $r_i$ by $(1-1/m)r_i$, $m \geq 2$, and re-indexing, we shall assume $\delta_i:=N(x_i)-r_i>0$ for all $i$. By the density of $G_N$, for each $i$ pick $u_i \in G_N$ close enough to $x_i/N(x_i)$ so that $z_i :=N(x_i)u_i \in B_N(x_i,\delta_i)$. Evidently, $\set{B_N(z_i,N(x_i))}{i \in \omega}$ is another BCP covering. We claim that $\set{j_{u_i}}{i \in \omega} \subseteq J_N$ is a total subset. Indeed, given $y \in S_N$, pick $i$ such that $N(z_i-y) < N(x_i)$. Then
 \[
  j_{u_i}(y) = j_{u_i}(z_i)-j_{u_i}(z_i-y) > N(x_i) - N(x_i) =0. \qedhere
  \]
\end{proof}

Our most general sufficient condition for the RBCP requires the countable separation property (CSP). A Banach space $X$ is said to have the CSP if every total subset $F \subseteq X^*$ contains a countable total subset \cite{Talponen2009}. Evidently, the CSP implies the separability of $(X^*,w^*)$.

\begin{theorem}\label{thm:main-thm-A}
Let \(X\) be a GDS with the CSP. Then \(X\) has the RBCP.
\end{theorem}

\begin{proof}
Fix an equivalent norm \(N\). Since \(X\) is a GDS, the smooth points \(G_N\) are dense in \(S_N\). Hence \(J_N\) is \(1\)-norming and therefore total. By the CSP, \(J_N\) contains a countable total subset. Proposition~\ref{lem:fixed-norm-characterization} shows that \(N\) has the BCP. Since \(N\) was arbitrary, \(X\) has the RBCP.
\end{proof}

The next three results show that, under some reasonably general conditions, a Banach space $X$ or a $C(K)$-space has the RBCP if and only if $(X^*,w^*)$ is separable or $K$ is separable, respectively. For these results we recall Corson's much studied property (C): we say that a Banach space $X$ has (C) if every collection of closed convex subsets of $X$ having empty intersection admits a countable subcollection having empty intersection \cite{Corson1961}. We also recall Pol's characterization of (C) in terms of a convex cousin of countable tightness: $X$ has (C) if and only if, given $A \subseteq X^*$ and $f \in \cl{A}^{w^*}$, there exists a countable set $A_0 \subseteq A$ such that $f \in \cl{\conv}^{w^*}(A_0)$ \cite[Theorem 3.4]{Pol1980}. Using this characterization, it is straightforward to prove that if $X$ has (C) and $(X^*,w^*)$ is separable, then $X$ has the CSP \cite[Proposition 3.1]{Talponen2009}.

\begin{corollary}
\label{cor:standard-criteria}
Let \(X\) be a GDS with (C). The following conditions are equivalent:
    \begin{enumerate}[label=\textup{(\roman*)}]
        \item \(X\) has the RBCP;
        \item some equivalent norm on \(X\) has the BCP;
        \item \((X^*,w^*)\) is separable.
    \end{enumerate}
\end{corollary}

\begin{proof}
The implication \textup{(i)} \(\Rightarrow\) \textup{(ii)} is immediate. Since \(X\) is a GDS, \textup{(ii)} \(\Rightarrow\) \textup{(iii)} follows from Proposition~\ref{lem:fixed-norm-characterization}. If \textup{(iii)} holds, then $X$ has the CSP from above. Theorem~\ref{thm:main-thm-A} now gives \textup{(i)}.
\end{proof}

Recall from the Introduction that \cite[Theorem 2.1]{LiuLiuLuZheng2022} characterizes both the BCP and UBCP of
\((C(K),\norm{\cdot}_\infty)\) in terms of the existence of a countable \(\pi\)-basis of \(K\). We now provide a characterization of the RBCP in $C(K)$-spaces, under some mild assumptions.

\begin{corollary}\label{cor:ck-general-criteria}
    Let \(K\) be a compact space such that \(C(K)\) is a GDS with (C). The following conditions are equivalent:
    \begin{enumerate}[label=\textup{(\roman*)}]
        \item \(C(K)\) has the RBCP;
        \item \((C(K),\norm{\cdot}_\infty)\) has the BCP;
        \item \((C(K),\norm{\cdot}_\infty)\) has the UBCP;
        \item \(K\) has a countable \(\pi\)-basis;
        \item \(K\) is separable.
    \end{enumerate}
\end{corollary}

\begin{proof}
    Condition \textup{(i)} implies \textup{(ii)}, and \cite[Theorem 2.1]{LiuLiuLuZheng2022} gives the equivalence of \textup{(ii)}, \textup{(iii)}, and \textup{(iv)}. Condition \textup{(iv)} implies \textup{(v)} by selecting one point from each member of a countable \(\pi\)-basis. If \textup{(v)} holds and \(D\subseteq K\) is countable and dense, then \(\{\delta_d:d\in D\}\) is total in \(C(K)^*\), so Corollary~\ref{cor:standard-criteria} gives \textup{(i)}.
\end{proof}

Our second characterization of the RBCP in $C(K)$-spaces operates under assumptions that are less general but easier to check because they involve purely topological properties of $K$. Recall that a compact space $K$ is Radon-Nikod\'ym compact if it is homeomorphic to a $w^*$-compact subset of a dual Banach space having the Radon-Nikod\'ym property. Every scattered compact space is totally disconnected and Radon-Nikod\'ym. A compact space $K$ is called countably tight if, given $A \subseteq K$ and $x \in \cl{A}$, there exists a countable set $A_0 \subseteq A$ such that $x \in \cl{A_0}$.

\begin{theorem}\label{thm:rn-characterization}
Let \(K\) be a totally disconnected Radon--Nikod\'ym compact space having countable tightness. The following conditions are equivalent:
\begin{enumerate}[label=\textup{(\roman*)}]
    \item \(C(K)\) has the RBCP;
    \item \((C(K),\norm{\cdot}_\infty)\) has the BCP;
    \item \((C(K),\norm{\cdot}_\infty)\) has the UBCP;
    \item \(K\) has a countable \(\pi\)-basis;
    \item \(K\) is separable.
\end{enumerate}
\end{theorem}

\begin{proof}
We verify the hypotheses of Corollary~\ref{cor:ck-general-criteria}. As $K$ is a totally disconnected Radon--Nikod\'ym compact space, it can be regarded as a closed subspace of a product \(\prod_{n\in\omega}K_n\) of countably many compact scattered spaces (see \cite[Theorem~3.6 and Corollary~3.7]{Arvanitakis2002}, wherein Arvanitakis credits this result to E.~Reznichenko).

In the proof of \cite[Theorem~3.6]{Arvanitakis2002}, each \(K_n\) is expressed as a continuous image of \(K\). As $K$ is countably tight, each $K_n$ is countably tight. We provide a quick proof of this fact. Fix \(n\in\omega\) and let $\map{q}{K}{K_n}$ be a continuous surjective map. Let $A \subseteq K_n$ and $y \in \cl{A}$. By surjectivity, compactness and continuity, $\cl{A} \subseteq q(\cl{q^{-1}(A)})$, so $y=q(x)$ for some $x \in \cl{q^{-1}(A)}$. Because $K$ is countably tight, $x \in \cl{B}$ for a countable set $B \subseteq q^{-1}(A)$. Therefore $q(B) \subseteq A$ and $y=q(x) \in q(\cl{B}) \subseteq \cl{q(B)}$, again by continuity.

Each \(K_n\) is scattered and countably tight, so $C(\prod_{n\in\omega}K_n)$ has (C) by \cite[Theorem 4.1]{Pol1980}. Since \(K\) is closed in this product, the restriction operator onto \(C(K)\) is a quotient map. Property (C) passes to quotients, and hence \(C(K)\) has (C).

According to \cite[Theorem~1.5.4]{fabian:97}, \(C(K)\) is Asplund generated, which in turn implies that it is weak Asplund by \cite[Section~1.3]{fabian:97}, and hence a GDS. Thus Corollary \ref{cor:ck-general-criteria} applies and gives the five stated equivalences.
\end{proof}

\section{Examples}\label{sec:examples}

We begin the section with what is, from one point of view, the `simplest' class of examples of non-separable $C(K)$-spaces having the RBCP (though from other perspectives it is a rich, complicated class). Recall that a family $\mathcal{A}$ of sets is called almost disjoint if $A \cap B$ is finite whenever $A,B \in \mathcal{A}$ are distinct. We shall write $\sqcup$ to denote unions of disjoint sets and we denote by $\alpha L$ the $1$-point compactification of a locally compact space $L$. In this section, natural numbers $n\in\omega$ will be regarded in the set-theoretic sense, i.e.~$n$ is the set of all natural numbers strictly less than $n$.

\begin{example}\label{ex:almost_disjoint}
Let $\mathcal{A}$ be an uncountable almost disjoint family of infinite subsets of $\omega$. We denote by $K_\mathcal{A}$ the standard Isbell-Mr\'owka compact scattered space associated with this family \cite{HernandezHernandezHrusak2018}:~$K_\mathcal{A}=\alpha(\omega \sqcup \mathcal{A})$, where every point in $\omega$ is isolated and each $A \in \mathcal{A}$ has the neighbourhood base of compact open sets $W(A,n):=\{A\} \cup (A\setminus n)$, $n\in\omega$. As $K_\mathcal{A}$ is scattered it is totally disconnected and Radon-Nikod\'ym.
 
It is clear that $K_\mathcal{A}$ is separable because $\omega \subseteq K_\mathcal{A}$ is dense. It is also easy to check directly that $K_\mathcal{A}$ is countably tight. Hence $C(K_\mathcal{A})$ has the RBCP by Theorem \ref{thm:rn-characterization}.
 
Finally, it is well-known that $K_\mathcal{A}$ is non-metrizable (for example, because it is compact, scattered and uncountable) and thus $C(K_\mathcal{A})$ is non-separable. This non-separability can also be witnessed directly by considering the uncountable family of indicator functions $f_A :=\mathbf{1}_{W(A,0)} \in C(K_\mathcal{A})$, $A \in \mathcal{A}$, which satisfy $\norm{f_A - f_B}_\infty=1$ whenever $A \neq B$.
\end{example}

Having established that there exist non-separable Banach spaces having the RBCP, which gives a negative answer to \cite[Question~5.49]{ChengKatoZhang2020}, we spend the rest of the section presenting more examples that show just how non-separable a Banach space can be while still possessing the RBCP. Roughly speaking, we show that separability can fail spectacularly.

The Isbell-Mr\'owka spaces $K_\mathcal{A}$ have Cantor-Bendixson height $3$. It turns out that the conclusion that $C(K_\mathcal{A})$ has the RBCP can be extended effortlessly to a much wider class, including all spaces $K$ having countable Cantor-Bendixson height. The authors of \cite{OrihuelaSmithTroyanski2012} introduce a weak separation property, labelled {\st}, as a tool for studying strict convexity in Banach spaces. A locally compact space $L$ has {\st} if there exist families $(\mathscr{U}_n)_{n\in\omega}$ of open subsets of $L$ having the property that, given distinct $x,y \in L$, one can find $n\in\omega$ such that $\{x,y\} \cap \bigcup\mathscr{U}_n\neq\varnothing$ and $\{x,y\} \cap U$ is at most a singleton for all $U \in \mathscr{U}_n$.

The class of compact spaces having {\st} is very broad:~it includes all compact spaces having countable Cantor-Bendixson height, and more generally all compact $\sigma$-discrete spaces, descriptive compact spaces and Gruenhage compact spaces (see e.g.~\cite{Raja2003,Smith2009,Smith2026,SmithTroyanski2010} as well as \cite{OrihuelaSmithTroyanski2012} for detailed information on these classes). Two results concerning {\st} are relevant to us here. Firstly, the 1-point compactification of any locally compact space having {\st} is countably tight \cite[Corollary 2.4]{OrihuelaSmithTroyanski2012}. Secondly, a scattered compact space $K$ has {\st} if and only if $C(K)$ admits an equivalent norm with a strictly convex dual norm \cite[Theorem 3.1]{OrihuelaSmithTroyanski2012}. The next result is an immediate consequence of Theorem \ref{thm:rn-characterization}.

\begin{corollary}\label{pr:star}
If the compact space $K$ is separable, totally disconnected, Radon--Nikod\'ym and has {\st}, then $C(K)$ has the RBCP.
\end{corollary}
 
In the next part of the section, we show in Example \ref{ex:strictly_convex} below that there are examples of compact scattered spaces $K$ such that $C(K)$ has the RBCP, yet $K$ fails to have {\st}, i.e.~$C(K)$ fails to admit an equivalent norm whose dual norm is strictly convex. For this we need to consider the idea of growths of $\omega$. A compact space $K$ is known as a growth of $\omega$ if there is a compactification $\gamma\omega$ of $\omega$ (with the discrete topology) such that $K$ is homeomorphic to $\gamma\omega\setminus\omega$. For example, the 1-point compactification of a discrete space of cardinality at most $\mathfrak{c}$ is a growth of $\omega$:~in this case $\gamma\omega=K_\mathcal{A}$ for some almost disjoint family $\mathcal{A}$ of subsets of $\omega$. Another standard example of a growth of $\omega$ is the ordinal $\omega_1+1$ which, with respect to the standard interval topology, is compact and scattered. This follows by considering a transfinite sequence of subsets of $\omega$ that is strictly increasing with respect to eventual dominance (see e.g.~\cite{Nyikos1990}).

It is clear that there are a number of topological properties which pass from any growth $K$ of $\omega$ to the corresponding compactification $\gamma\omega$, and vice-versa. This applies in particular to the property of being scattered, or having {\st}, or having countable tightness. The examples in Example \ref{ex:strictly_convex}, facilitated by Proposition \ref{prop:AB-system} and Corollary \ref{cor:tree_growth} below, will be compactifications $\gamma\omega$, where $\gamma\omega\setminus\omega$ is homeomorphic to the $1$-point compactification of a tree (with respect to the interval topology) satisfying some mild conditions. That such $1$-point compactifications of trees can be growths of $\omega$ can be viewed as a simultaneous generalisation of the two standard examples of growths given above.

Before presenting Proposition \ref{prop:AB-system}, we need to introduce some further terminology and prove a combinatorial lemma. Recall that a tree, in the classical (not descriptive) set-theoretic sense, is a partially ordered set $(T,\preccurlyeq)$ having the property that the set of predecessors of any $t\in T$ is well-ordered. Trees are thus natural generalisations of ordinal numbers. Anticipating what will come later and for convenience, we shall always assume $T \cap \omega=\varnothing$. We will place a root $0_T \notin T$ below all the minimal nodes of $T$, i.e.~$0_T \prec t$ for all $t \in T$. Given $t\in T\cup\{0_T\}$, let $t^+$ denote the set of immediate successors of $t$ in $T$, that is, the set of minimal elements of $\set{u \in T}{t \prec u}$ (thus $0_T^+$ is the set of minimal elements of $T$). Further, denote by
\[
T^+ = \bigcup_{t\in T\cup\{0_T\}} t^+
\]
the set of minimal or successor elements of $T$; any element of $T$ that is not in $T^+$ will be called a limit element.

We equip each tree $T$ with its standard interval topology. Given $r \in T \cup \{0_T\}$ and $t \in T$, define the interval $(r,t]=\set{s \in T}{r \prec s \preccurlyeq t}$. The interval topology of $T$ has as its base all (compact) sets of the form $(r,t]$, with $r$ and $t$ as above. This topology is Hausdorff if and only if every non-empty totally ordered subset of $T$ has at most one minimal upper bound; we shall call any such tree with this property a Hausdorff tree and will insist that every tree under consideration is Hausdorff. The interval topology is clearly locally compact, and the fact that $(T,\preccurlyeq)$ is well-founded ensures that it is also scattered. With respect to this topology, $T^+$ is the set of isolated points.

Given $u,v \in T$, we define $u \wedge v$ to be the greatest element of $T\cup\{0_T\}$ less than or equal to both $u$ and $v$. This definition is sound because $u \wedge v$ is the (necessarily unique) minimal upper bound of the non-empty totally ordered set
\[
 \set{t \in T^+ \cup \{0_T\}}{t \preccurlyeq u,v}.
\]
When $u \perp v$, that is, when $u,v$ are incomparable, we have $u \wedge v \prec u,v$. We have $u \wedge v = 0_T$ if and only if $(0_T,u]\cap(0_T,v]=\varnothing$. We shall denote by $\htop(t)$ the order type of the set $(0_T,t)$ of strict predecessors of $t$ in $T$. For technical convenience, set $\htop(0_T)=-1$.

Before revealing our combinatorial lemma, Lemma \ref{lm:AB-system}, we need some terminology and an observation regarding eventual dominance. Given sets $A,B \subseteq \omega$, write $A \ssubstar B$ to mean $A\setminus B$ is finite, and $A \sstar B$ when $A\setminus B$ is finite and $B\setminus A$ is infinite. Now suppose we have sets $A_n,B_n \subseteq \omega$, $n\in\omega$, such that
\[
 A_0 \sstar A_1 \sstar A_2 \sstar \ldots \sstar B_2 \sstar B_1 \sstar B_0.
\]
Then it is well-known, using standard diagonal arguments, that there exist $A,B \subseteq \omega$ satisfying
\begin{equation}\label{eqn_gap}
 A_n \sstar A \sstar B \sstar B_n, \quad n \in \omega.
\end{equation}
For instance, the reader can verify that
\[
 A = \bigcup_{n \in \omega} A_n \cap \bigcap_{i\leq n} B_i \quad\text{and}\quad B=A \cup \set{k_n}{n \in \omega},
\]
where the $k_n \in \bigcap_{i \leq n} B_i\setminus A$ are distinct, satisfy \eqref{eqn_gap}.

\begin{lemma}\label{lm:AB-system}
 Let $T$ be a tree having cardinality at most $\mathfrak{c}$, and such that $\htop(t)<\omega_1$ whenever $t \in T$. Then there exists a family $(A_t,B_t)_{t\in T \cup \{0_T\}}$ of pairs of subsets of $\omega$ satisfying
 \begin{enumerate}[label=\textup{(\roman*)}]
  \item $A_t \sstar B_t$;
  \item $A_t \sstar A_u$ and $B_u \sstar B_t$ whenever $t \prec u$;
  \item $B_u \cap B_v \ssubstar A_{u \wedge v}$ whenever $u \perp v$.
 \end{enumerate}
\end{lemma}

\begin{proof}
We proceed by transfinite recursion on $\htop(t)$, where $-1 \leq \htop(t)<\omega_1$. To begin, set $A_{0_T}=\varnothing$ and $B_{0_T}=\omega$. Now let $0 \leq \alpha < \omega_1$ and assume properties (i)--(iii) hold for all $t,u,v \in T \cup \{0_T\}$ satisfying $\htop(t),\htop(u),\htop(v) < \alpha$. There are two cases to consider. The first case is that $\alpha=0$ or $\alpha$ is a successor ordinal. Then $\alpha=\xi+1$, where $-1 \leq \xi <\omega_1$. We will define $A_t$ and $B_t$ whenever $\htop(s)=\xi$ and $t \in s^+$. This ensures that $A_t$ and $B_t$ are defined whenever $\htop(t)=\alpha$. Let $s \in T \cup \{0_T\}$ such that $\htop(s)=\xi$. Because $B_s\setminus A_s$ is infinite by (i) and $s^+$ has cardinality at most $\mathfrak{c}$, there exists an almost disjoint family $(D_t)_{t \in \{s\} \cup s^+}$ of infinite subsets of $B_s\setminus A_s$. For each $t \in s^+$, select an infinite set $C_t \subseteq D_t$ such that $D_t\setminus C_t$ is also infinite. Then define $A_t = A_s \cup C_t$ and $B_t = A_s \cup D_t$ whenever $t \in s^+$. The second case is that $\alpha$ is a limit ordinal. To manage this case, first fix a strictly increasing sequence of ordinals $(\xi_n)_{n\in\omega}$ that converges to $\alpha$. Given $t \in T$ with $\htop(t)=\alpha$, there exists a unique $s_n \in (0_T,t]$ such that $\htop(s_n)=\xi_n$. By property (ii), we have
 \[
  A_{s_0} \sstar A_{s_1} \sstar A_{s_2} \sstar \ldots \sstar B_{s_2} \sstar B_{s_1} \sstar B_{s_0}.
 \]
 Pick $A,B \subseteq \omega$ satisfying \eqref{eqn_gap} and set $A_t=A$, $B_t=B$.

 Now we prove that (i)--(iii) hold for all $t,u,v \in T \cup \{0_T\}$ satisfying $\htop(t),\htop(u),\htop(v) \leq \alpha$. That (i) holds is immediate from the above. Next, we tackle property (ii). Let $r \prec t$, where $\htop(t) \leq \alpha$. If $\htop(t)<\alpha$ then there is nothing to prove, so we assume $\htop(t)=\alpha$. If $\alpha=0$ or $\alpha$ is a successor ordinal, let $\alpha=\xi+1$ and $s \in T \cup \{0_T\}$ such that $t \in s^+$. Either $r=s$ or $r \prec s$, so by applying the inductive hypothesis in the second case if necessary, we have $A_r \ssubstar A_s \sstar A_s \cup C_t = A_t$ and $B_t = A_s \cup D_t \sstar B_s \ssubstar B_r$ (note $B_s\setminus B_t \supseteq D_s\setminus D_t$, which is infinite). Now assume that $\alpha$ is a limit ordinal. Then we can pick $n\in\omega$ such that $r \prec s_n \prec t$. By inductive hypothesis and the above, we have $A_r \sstar A_{s_n} \sstar A_t$ and $B_t \sstar B_{s_n} \sstar B_r$. Finally, we address property (iii). Assume $u \perp v$ and $\htop(u),\htop(v) \leq \alpha$. Set $t=u \wedge v \prec u,v$ and then let $u' \in t^+ \cap (0_T,u]$ and $v' \in t^+ \cap (0_T,v]$ (these elements are unique). From above, $(B_{u'} \cap B_{v'})\setminus A_t = (A_t \cup (D_{u'} \cap D_{v'}))\setminus A_t$ is finite, so $B_{u'} \cap B_{v'} \ssubstar A_t$. As $u' \preccurlyeq u$ and $v' \preccurlyeq v$, by (ii) (which we have just proved applies now to $u$ and $v$), we have $B_u \cap B_v \ssubstar B_{u'} \cap B_{v'} \ssubstar A_t$. This concludes the recursion.
\end{proof}

Now we can present Proposition \ref{prop:AB-system}, which elaborates on the way $\omega_1+1$ is shown to be a growth of $\omega$ in \cite{Nyikos1990}.

\begin{proposition}\label{prop:AB-system}
Let $T$ be a tree having cardinality at most $\mathfrak{c}$, and such that $\htop(t)<\omega_1$ whenever $t \in T$. Then $\alpha T$ is a growth of $\omega$.
\end{proposition}

\begin{proof}
 Recalling the assumption that $T \cap \omega = \varnothing$, we shall first define a locally compact scattered Hausdorff topology on $\omega \sqcup T$ which extends the topology on $T$ and with respect to which $\omega$ is the set of isolated points. Take a family $(A_t,B_t)_{t \in T \cup \{0_T\}}$ of pairs of subsets of $\omega$ satisfying Lemma \ref{lm:AB-system} (i)--(iii). Given $s,t \in T \cup \{0_T\}$, $s \prec t$, and $n\in\omega$, define
 \[
  W(s,t,n) = (s,t] \cup A_t\setminus(A_s \cup n).
 \]
 We claim that the sets $W(s,t,n)$ (together with the singletons $\{n\}$, $n\in\omega$) form a basis of a Hausdorff topology on $\omega \sqcup T$. Firstly, we verify that these sets form a basis of a scattered topology which extends that of $T$. Given that each $n\in\omega$ is isolated, it is enough to prove that if $t \in W(r_1,u_1,n_1) \cap W(r_2,u_2,n_2) \cap T$, then there exists $s \prec t$ and $n \in \omega$ such that $W(s,t,n) \subseteq W(r_i,u_i,n_i)$, $i=1,2$. Given such $t$, we have $r_1,r_2 \prec t$, so they are comparable and we can set $s=\max\{r_1,r_2\} \prec t$. Given $i=1,2$, we have $(s,t] \subseteq (r_i,u_i]$ and, by Lemma \ref{lm:AB-system} (ii), $ A_t\setminus A_{u_i}$ and $ A_{r_i}\setminus A_s$ are finite, which implies $A_t\setminus A_s \ssubstar A_{u_i}\setminus A_{r_i}$. This means we can pick $n\in\omega$ large enough to satisfy
 \[
 A_t\setminus (A_s \cup n) \subseteq A_{u_i}\setminus A_{r_i}, \quad i=1,2.
 \]
We conclude that $W(s,t,n) \subseteq W(r_i,u_i,n_i)$, $i=1,2$, as required. Hence we have a basis of a topology on $\omega\sqcup T$. Because $W(s,t,n) \cap T = (s,t]$, it is clear that this topology extends that of $T$. As $T$ is scattered and each point of $\omega$ is isolated, it follows that $\omega\sqcup T$ is also scattered.

Next, we demonstrate that $\omega\sqcup T$ is Hausdorff. Given the definition of the $W(s,t,n)$ and the fact that each $n\in\omega$ is isolated, it is easy to see that we only need to verify this when $u,v \in T$ are distinct. There are two cases. If $u \prec v$ then
\[
 W(0_T,u,0) \cap W(u,v,0) = ((0_T,u] \cup A_u) \cap ((u,v] \cup A_v\setminus A_u) = \varnothing.
\]
Otherwise, $u \perp v$, meaning that $A_u \cap A_v \ssubstar B_u \cap B_v \ssubstar A_t$ by Lemma \ref{lm:AB-system} (i) and (iii), where $t:=u \wedge v \prec u,v$. Observe that $(t,u] \cap (t,v] = \varnothing$, so given $n\in\omega$,
\[
 W(t,u,n) \cap W(t,v,n) = (A_u \cap A_v)\setminus(A_t \cup n),
\]
which is empty for a large enough choice of $n$. Hence $\omega\sqcup T$ is Hausdorff.

Finally, we show that the $W(s,t,n)$ are compact. We use induction on $\htop(t)$. Let $\alpha < \omega_1$ and assume $W(s,t,n)$ is compact whenever $s \prec t$, $n\in\omega$ and $\htop(t) < \alpha$. Now let $s \prec t$, $n\in \omega$ and $\htop(t)=\alpha$, and let $\mathscr{U}$ be an open cover of $W(s,t,n)$ of sets open in $\omega\sqcup T$. Pick $U \in \mathscr{U}$ such that $t\in U$, and then $r \prec t$ and $m \in \omega$ such that $W(r,t,m) \subseteq U$. If $r \preccurlyeq s$ then $W(s,t,n)\setminus U \subseteq W(s,t,n)\setminus W(r,t,m)$ is finite, thus $W(s,t,n)$ can be covered by finitely many elements of $\mathscr{U}$. Otherwise, $s \prec r$, and we can pick $p \in \omega$ large enough so that $W(s,r,p) \subseteq W(s,t,n)$. By inductive hypothesis, $W(s,r,p)$ can be covered by finitely many elements of $\mathscr{U}$, and as $W(s,t,n)\setminus(W(s,r,p) \cup W(r,t,m))$ is finite, the same can be said of $W(s,t,n)$. This concludes the induction.

Hence $\alpha(\omega\sqcup T)$ is a compactification of $\omega$ and there is a natural homeomorphism from $\alpha T$ onto $\alpha(\omega\sqcup T)\setminus \omega$. Therefore $\alpha T$ is a growth of $\omega$.
\end{proof}

\begin{corollary}\label{cor:tree_growth}
Let $T$ be a tree having cardinality at most $\mathfrak{c}$ that does not admit any uncountable totally ordered subsets. Then $\alpha T$ is homeomorphic to $\gamma\omega\setminus\omega$ for some countably tight scattered compactification $\gamma\omega$ of $\omega$.
\end{corollary}

\begin{proof}
By Proposition \ref{prop:AB-system}, $\alpha T$ is homeomorphic to $\gamma\omega\setminus\omega$ for some scattered compactification $\gamma\omega$ of $\omega$. The countable tightness of $\gamma\omega$ is an immediate consequence of that of $\alpha T$, which follows by a standard argument that we include for completeness. Let $x \in \cl{A} \subseteq \alpha T$. If $x=t \in T$ then $A\cap(0_T,t]$ is countable and $x \in \cl{A\cap(0_T,t]}$. Hereafter assume $x=\infty$ and $A$ is uncountable. Given that $T$ admits no uncountable totally ordered subsets, a standard application of the Dushnik-Miller theorem yields that $A$ must admit a countably infinite antichain $B$, and by virtue of the topology on $\alpha T$, we have $x=\infty \in \cl{B}$.
\end{proof}

At last, we can produce examples of compact spaces $K$ such that $C(K)$ has the RBCP, but $C(K)$ admits no equivalent norm having a strictly convex dual norm.

\begin{example}\label{ex:strictly_convex}
Let $T$ be any tree having cardinality at most $\mathfrak{c}$ that does not admit any uncountable totally ordered subsets, and which also fails {\st}. An example of this is $T=\sigma Y$ (see e.g.~\cite[pp.~405--406]{SmithTroyanski2010}). By Corollary \ref{cor:tree_growth}, $\alpha T$ is homeomorphic to $\gamma\omega\setminus\omega$ for some countably tight scattered compactification $\gamma\omega$ of $\omega$. By Theorem \ref{thm:rn-characterization}, $C(\gamma\omega)$ has the RBCP. On the other hand, $\gamma\omega$ fails {\st} because $T$ fails {\st}. By \cite[Proposition 3.2]{OrihuelaSmithTroyanski2012}, $C(\gamma\omega)$ admits no equivalent norm having a strictly convex dual norm.
\end{example}

We continue in the spirit of finding Asplund $C(K)$-spaces that have the RBCP but behave badly with respect to renormings. The next example makes use of the forcing construction of Brech and Koszmider in \cite{BrechKoszmider2011}. Recall that a topological space is called hereditarily separable if every subspace of it is separable. Hereditarily separable spaces are clearly countably tight.

\begin{example}\label{ex:Brech_Koszmider}
By Theorem~3.2 and Corollaries~3.3 and 3.5 in \cite{BrechKoszmider2011}, there is, in a forcing extension, a compact scattered space \(K_{\mathrm{BK}}\) that is hereditarily separable (moreover, every finite power of it is hereditarily separable) and \(\dens C(K_{\mathrm{BK}})=\aleph_2\). The same results show that, in this model, \(C(K_{\mathrm{BK}})\) admits no equivalent Fr\'echet smooth norm and contains no uncountable biorthogonal system. On the other hand, as $K_{\mathrm{BK}}$ is scattered, separable and countably tight, $C(K_{\mathrm{BK}})$ has the RBCP by Theorem \ref{thm:rn-characterization}.
\end{example}

We conclude the section by finding an even more extreme example of a $C(K)$-space that has the RBCP. This one fails even to be weak Asplund. In \cite{kalenda:99}, Kalenda defines a class of compact spaces modelled on the classical Alexandrov-Urysohn `double arrow space' or `split interval'. He defines this class in terms of arbitrary perfect compact subsets $K\subseteq \R$, but here we will assume $K=[0,1]$. Let $A \subseteq (0,1)$ be arbitrary and set
\[
 L_A = ((0,1] \times \{0\}) \cup ((\{0\} \cup A) \times \{1\}).
\]
We equip $L_A$ with the standard lexicographic ordering:~$(x_1,x_2) \leq (y_1,y_2)$ if and only if $x_1 < y_1$ or $x_1=y_1$ and $x_2 \leq y_2$. Then the induced order topology on $L_A$ is compact and separable (indeed, hereditarily separable) \cite[Proposition 2]{kalenda:99}. When $A=(0,1)$ we recover the usual split interval. We will use Corollary \ref{cor:ck-general-criteria} to show that, in a particularly important case, $C(L_A)$ has the RBCP.

That  following result generalises a fact that is well-known in the case $A=(0,1)$.

\begin{proposition}\label{prop:split-exact-sequence}
 Let $A \subseteq (0,1)$ be arbitrary. The bounded linear operator
  \[
   \map{T}{C(L_A)}{c_0(A)}, \quad Tf(a) = f(a,1)-f(a,0), \quad a \in A,
  \]
 is surjective and its kernel is linearly isometric to $C[0,1]$. Consequently, $C(L_A)$ has (C).
\end{proposition}

\begin{proof}
That $T$ has the given properties follows from an adaptation of the first part of the proof of \cite[Theorem 2.3.1]{fabian:97}; for this reason we omit the proof. Given that $C[0,1]$ and $c_0(A)$ both have (C), we can apply \cite[Proposition 3.2 (1)]{Pol1980} to conclude that $C(L_A)$ also has (C).
\end{proof}

In \cite{moors:somasundaram:06}, Moors and Somasundaram, with help from Fremlin and Preiss, constructed a set $A \subseteq (0,1)$ such that the resulting space $C(L_A)$ is a GDS but is not weak Asplund (thus solving a problem of Larman and Phelps from 1979). The result of Preiss, Phelps, and Namioka \cite{preiss:phelps:namioka:90} implies that \(C(L_A)\) admits no equivalent G\^ateaux smooth norm. Despite this atrocious behaviour, as $C(L_A)$ is a GDS, Proposition \ref{prop:split-exact-sequence} and Corollary \ref{cor:ck-general-criteria} yield immediately the final result of the section.

\begin{corollary}\label{co:Moors-Somasundaram}
 The Moors-Somasundaram space $C(L_A)$ has the RBCP.
\end{corollary}

\section{Problems and remarks}\label{sec:remarks_problems}

Our first problem concerns scattered spaces and renormings. As we saw in Example \ref{ex:strictly_convex}, there exists a compact separable scattered space $K$ such that $C(K)$ has the RBCP but no equivalent norm whose dual norm is strictly convex. According to Example \ref{ex:Brech_Koszmider}, it is consistent that there is a compact separable scattered space $K$ such that $C(K)$ has the RBCP but no equivalent Fr\'echet smooth norm. It is natural to ask whether there is a ZFC version of Example \ref{ex:Brech_Koszmider}.

Haydon established that there exists a compact scattered (non-separable) space $K$ such that $C(K)$ admits no equivalent strictly convex norm or G\^ateaux smooth norm \cite{Haydon:1992}. Here, $K=\alpha T$, where $T$ is a tree having cardinality $\mathfrak{c}$ and no uncountable branches, so it is countably tight. While $K$ itself is not separable, according to Proposition \ref{prop:AB-system} $K$ is homeomorphic to $\gamma\omega\setminus\omega$ for some compactification $\gamma\omega$ of $\omega$. However, the arguments applied to $K$ in \cite{Haydon:1992} to show the non-existence of equivalent strictly convex norms and G\^ateaux smooth norms on $C(K)$ are similar (the latter is an elaboration of the former), and any attempt to adapt them to $\gamma\omega$ is doomed to fail because the dual of $C(\gamma\omega)$ is $w^*$-separable, and it is well-known that any such Banach space admits a bounded linear injective operator into $\ell_2$, and thus admits an equivalent strictly convex norm.

These observations prompt the following problem, part of which can be stated without reference to the RBCP and may be of independent interest.

\begin{problem}
 Does there exist, in ZFC, a compact, scattered and separable space $K$ such that $C(K)$ admits no equivalent Fr\'echet smooth norm? If so, can we also ensure that $C(K)$ has the RBCP? Can we go further and replace Fr\'echet smooth above by G\^ateaux smooth?
\end{problem}

Again, as stated in the Introduction, $C(\beta\omega) \equiv \ell_\infty$ has the BCP but fails the RBCP. Clearly $\beta\omega$ is not scattered. If $K$ is compact and scattered then it has a countable $\pi$-base if and only if it is separable (if $K$ is separable then its set $I$ of isolated points must be countable, whence $\set{\{x\}}{x \in I}$ is a countable $\pi$-base). Thus $C(K)$ has the BCP and UBCP in these circumstances. As we have seen, many such spaces also have the RBCP. However, we were unable to establish whether or not $C(K)$ has the RBCP whenever $K$ is scattered and separable. This prompts our second problem, which is connected to the first.

\begin{problem}
Let $K$ be compact, scattered and separable. Does $C(K)$ have the RBCP?
\end{problem}

One can ask as a special case whether there is, in ZFC, a compactification $\gamma\omega$ of $\omega$ such that $\gamma\omega\setminus \omega$ is homeomorphic to $\omega_1+1$ and $C(\gamma\omega)$ fails the RBCP.

\section*{Acknowledgments}
The authors are grateful to the Institute for Theoretical Sciences at Westlake University for hosting the 2026 Summer School in ``Lipschitz-free Spaces and Metric Geometry". This project started amid this stimulating environment.
We also thank Thierry De Pauw and the organizing team for their warm hospitality and support throughout the programme.

\bibliographystyle{amsplain}
\bibliography{cite}

\end{document}